\documentclass{amsart}
\usepackage{amsmath}
\usepackage{amssymb}
\usepackage{amsthm}
\usepackage{graphicx}
\usepackage{tikz}
\usepackage{url}
\newtheorem{thm}{Theorem}
\newtheorem{lem}[thm]{Lemma}
\newtheorem{prop}[thm]{Proposition}
\newtheorem{cor}[thm]{Corollary}
\newcommand{\INT}{\operatorname{int}}
\newcommand{\AFF}{\operatorname{aff}}
\newcommand{\DIST}{\operatorname{dist}}

\newcommand{\VECT}[2]{\left({\begin{array}{c} #1 \\ #2 \end{array}}\right)}

\begin{document}

\title{Self-affine quadrangles}

\author{Christian Richter}
\author{Felix Zimmermann}

\address{Institute for Mathematics, Friedrich Schiller University, 07737 Jena, Germany}
\email{christian.richter@uni-jena.de}
\email{zimmermann.f@uni-jena.de}

\date{\today}

\begin{abstract}
A quadrangle in the Euclidean plane is called $n$-self-affine if it has a dissection into $n$ affine images of itself. All convex quadrangles are known to be $n$-self-affine for every $n \ge 5$. The only $2$-self-affine convex quadrangles are trapezoids. Here we characterize all $3$-self-affine convex quadrangles, obtaining 
$5$ one-parameter families and $13$ singular examples of affine types. This way we reduce the quest for all $n$-self-affine convex quadrangles to the open case $n=4$.

In addition, we show that there are $n$-self-affine non-convex quadrangles for all $n \ge 3$, but not for $n=2$.
\end{abstract}


\maketitle


\section{Introduction}

A topological disc $D$ in the Euclidean plane $\mathbb R^2$ is called \emph{$n$-self-affine}, $n \in \{2,3,\ldots\}$, if it admits a dissection into $n$ images $\alpha_1(D),\ldots,\alpha_n(D)$ of itself under suitable affine transformations $\alpha_1,\ldots,\alpha_n$ of $\mathbb R^2$. Here a \emph{topological disc} $D$ is a homeomorphic image of a closed circular disc. It is \emph{dissected} into $\alpha_1(D),\ldots,\alpha_n(D)$ if $D=\alpha_1(D)\cup\ldots\cup\alpha_n(D)$ and the \emph{interiors} $\INT(\alpha_1(D)),\ldots,\INT(\alpha_n(D))$ are mutually disjoint. The disc $D$ is called \emph{self-affine} if it is $n$-self-affine for some $n \ge 2$.

The concept of self-affinity goes back to fractal geometry, where the affine transformations are supposed to be contractions, whereas more general compact sets are allowed instead of discs \cite[Section 9.4]{falconer1997}. It generalizes \emph{self-similarity}, where the transformations are restricted to similarities \cite[Section 9.2]{falconer1997}. Self-similar discs include such aesthetic examples as the golden bee \cite[Section 4.8]{scherer1987} that seems to go back to R.\ Ammann \cite[Section 10.4]{gruenbaum_shephard1987}, the snail of S.W.\ Golomb \cite[p.\ 406]{golomb1964} (apparently sometimes called elf boot) and the twindragon \cite[p.\ 66]{mandelbrot1982}, \cite[p.\ 30]{edgar1990}, all being non-convex.

We are mainly interested in the case of convex discs. All convex self-affine discs are necessarily polygons \cite[Proposition~1]{richter2012} and, even more restrictive, triangles, quadrangles or pentagons \cite[Satz 1]{hertel2000}, \cite[Corollary]{richter2012}. Clearly, every triangle is $n$-self-affine for every $n \ge 2$, as it can be dissected into $n$ triangles. In contrast with that, the situation for pentagons appears highly non-trivial. There exists an example of a self-affine convex pentagon, more precisely of a $27$-self-affine convex pentagon \cite[Proposition~4]{hertel_richter2010}. But pentagons close to the regular one are known to be not self-affine \cite[Proposition~3]{hertel_richter2010}, \cite[Theorem~1]{blechschmidt_richter2015}. 

Here we study self-affine quadrangles. So far it is known that the only $2$-self-affine convex quadrangles are the trapezoids \cite[Satz~3]{hertel2000}, \cite[Theorem~1(iii)]{richter2024+}. On the other hand, every convex quadrangle is $n$-self-affine for every $n \ge 5$ \cite[Theorem~2]{richter2024+}, where the case $n=5$ goes back to A.\ P\'or \cite[Proposition~1]{hertel_richter2010}. The main result of the present paper is the complete characterization of all $3$-self-affine convex quadrangles. The search for all $4$-self-affine convex quadrangles is an ongoing project. Moreover, we give first examples of self-affine non-convex quadrangles.

The paper is organized as follows. 

We have to classify all convex quadrangles up to affine equivalence, i.e., congruence under affine transformations. In Section~\ref{sec:parameters} we shall characterize all affine types by two real parameters. We consider this as the natural parametrization for the present paper, being aware that other parametrizations are used in other contexts, see \cite[Section~2]{hertel2000}, \cite[Section~2]{hertel_richter2010}, \cite[Section~4]{richter2024+}. 

Self-affine convex quadrangles are already completely characterized for the particular case that the underlying dissection is a so-called glass-cut dissection \cite{richter2024+}. This is explained and summarized in Section~\ref{sec:glass-cut}. In particular, the parametrization from \cite{richter2024+} is translated into the natural one of the present paper.

Section~\ref{sec:non-glass-cut} concerns the situation of non-glass-cut dissections. This is based on a computer algebraic approach from the thesis \cite{zimmermann2023}. Then Section~\ref{sec:union} brings the above solutions together and illustrates the overall result.

Finally, Section~\ref{sec:non-convex} gives first insights into self-affinity of non-convex quadrangles. It is shown that there exist $n$-self-affine non-convex quadrangles if and only if $n \ge 3$.


\section{Parametrizing affine types of convex quadrangles\label{sec:parameters}}

The following parametrization has been introduced independently in \cite[Section~4]{boerner2020} and \cite[Section~4]{zimmermann2023}.
We call it the \emph{natural parametrization} in the present paper.

Let $Q$ be a convex quadrangle with vertices $v_1,v_2,v_3,v_4$. We pick three consecutive vertices of $Q$, say $v_1,v_2,v_3$. Then there is a unique affine transformation $\alpha_{123}$ mapping them following their order onto $\alpha_{123}(v_1)=\VECT{0}{1}$, $\alpha_{123}(v_2)=\VECT{0}{0}$ and $\alpha_{123}(v_3)=\VECT{1}{0}$, see Figure~\ref{fig:Q[x,y]}. 
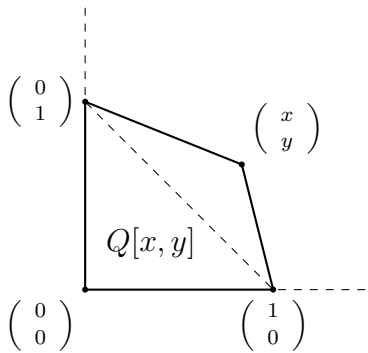
\begin{figure}
\begin{center}
\begin{tikzpicture}[xscale=2.5,yscale=2.5]

\draw[thick]
  (0,1)--(0,0)--(1,0)--(.833,.666)--cycle
	;

\draw[dashed]
  (0,1.5)--(0,1)--(1,0)--(1.5,0)
	;

\fill
  (0,1) circle (.016) node[left] {\footnotesize $\VECT{0}{1}$} 
  (0,0) circle (.016) node[below left] {\footnotesize $\VECT{0}{0}$} 
  (1,0) circle (.016) node[below] {\footnotesize $\VECT{1}{0}$} 
  (.833,.666) circle (.016) node[above right] {\footnotesize $\VECT{x}{y}$}
	(.35,.25) node {\Large $Q[x,y]$}
	;

\end{tikzpicture}
\end{center}
\caption{The affine type $Q[x,y]$.\label{fig:Q[x,y]}}
\end{figure}
The image $\alpha_{123}(v_4)=\VECT{x}{y}$ determines $Q$ up to affine congruence. We say that $Q$ is of \emph{affine type} $Q[x,y]$. Since $Q$ is convex, we have 
\begin{equation}\label{eq:full_domain}
x>0,\quad y>0, \quad x+y > 1,
\end{equation}
illustrated with dashed lines in Figure~\ref{fig:Q[x,y]}.
Conversely, every vector $\VECT{x}{y}$ with property \eqref{eq:full_domain} determines an affine type. We shall use the notation $Q[x,y]$ for the class of all quadrangles of the affine type $Q[x,y]$, but sometimes also for its particular representative from Figure~\ref{fig:Q[x,y]}.

The pair of parameters $(x,y)$ is not uniquely determined by $Q$, because there are eight possible choices of three consecutive vertices of $Q$, namely $v_i,v_j,v_k$ with $(i,j,k) \in \{(1,2,3),(2,3,4),(3,4,1),(4,1,2),(3,2,1),(4,3,2),(1,4,3),(2,1,4)\}$. For example, the map
\[
\varphi_{(x,y)}\VECT{\xi}{\eta}=\left(\begin{array}{cc} 
0 & \frac{1}{y} \\
-1 & \frac{x-1}{y}
\end{array}\right)\VECT{\xi}{\eta}+\VECT{0}{1}
\]
gives $\varphi_{(x,y)}\VECT{0}{0}=\VECT{0}{1}$, $\varphi_{(x,y)}\VECT{1}{0}=\VECT{0}{0}$ and $\varphi_{(x,y)}\VECT{x}{y}=\VECT{1}{0}$. Then the composition $\alpha_{234}=\varphi_{(x,y)} \circ \alpha_{123}$ yields $\alpha_{234}(v_2)=\VECT{0}{1}$, $\alpha_{234}(v_3)=\VECT{0}{0}$ and $\alpha_{234}(v_4)=\VECT{1}{0}$. This way we get a second parametrization $Q[x',y']$ of $Q$ where
\[
\VECT{x'}{y'}=\alpha_{234}(v_1)=\varphi_{(x,y)}\VECT{0}{1}=\VECT{\frac{1}{y}}{\frac{x+y-1}{y}}.
\]
Similarly, by $\alpha_{341}=\varphi_{(x',y')}\circ\alpha_{234}$ and $\alpha_{412}=\varphi_{(x'',y'')}\circ\alpha_{341}$, we get the parametrizations $Q[x'',y'']$ and $Q[x''',y''']$ with
\begin{equation}\label{eq:sub2}
\VECT{x''}{y''}=\alpha_{341}(v_2)=\VECT{\frac{y}{x+y-1}}{\frac{x}{x+y-1}} \;\text{ and }\; \VECT{x'''}{y'''}=\alpha_{412}(v_3)=\VECT{\frac{x+y-1}{x}}{\frac{1}{x}}.
\end{equation}
Finally, if $\psi\VECT{\xi}{\eta}=\VECT{\eta}{\xi}$, we get $\alpha_{kji}=\psi\circ\alpha_{ijk}$, this way obtaining the additional parametrizations $Q[\overline{x},\overline{y}]$, $Q[\overline{x}',\overline{y}']$, $Q[\overline{x}'',\overline{y}'']$ and  $Q[\overline{x}''',\overline{y}''']$ with
\begin{align}
\VECT{\overline{x}}{\overline{y}}&=\VECT{y}{x},&
\VECT{\overline{x}'}{\overline{y}'}&=\VECT{\frac{x+y-1}{y}}{\frac{1}{y}},\label{eq:sub3}\\
\VECT{\overline{x}''}{\overline{y}''}&=\VECT{\frac{x}{x+y-1}}{\frac{y}{x+y-1}},&
\VECT{\overline{x}'''}{\overline{y}'''}&=\VECT{\frac{1}{x}}{\frac{x+y-1}{x}}.\label{eq:sub4}
\end{align}

One can always pick a unique parametrization $Q[x,y]$ of $Q$ whose parameters $(x,y)$ are in the region
\[
\mathcal P=\left\{(x,y) \in \mathbb R^2: x+y >1,\, y \le 1,\, x \le y\right\},
\]
see Figure~\ref{fig:domains}. 
\begin{figure}
\begin{center}
\begin{tikzpicture}[xscale=2,yscale=2]

\draw[thick]
  (0,1)--(2.5,1)
	(1,0)--(1,2.5)
	(0,2)--(2,0)
	(.5,.5)--(2.5,2.5)
	;

\draw[dashed]
  (0,2.5)--(0,1)--(1,0)--(2.5,0)
	;

\fill
  (.833,.666) circle (.02)
	(1.5,.75) circle (.02)
	(1.333,1.666) circle (.02)
	(.6,1.2) circle (.02)
	(.666,.833) circle (.02)
	(.75,1.5) circle (.02)
	(1.666,1.333) circle (.02)
	(1.2,.6) circle (.02)
	;
	
\draw
  (0,2.4) node[right] {\footnotesize $x=0$}
  (1,2.4) node[right] {\footnotesize $x=1$}
  (2.5,2.4) node[right] {\footnotesize $y=x$}
  (2.5,0) node[right] {\footnotesize $y=0$}
  (2.5,1) node[right] {\footnotesize $y=1$}
	(.75,.3) node[below left] {\footnotesize $x+y=1$}
	(1.75,.2) node[above right] {\footnotesize $x+y=2$}
  (.833,.45) node {\Large $\overline{\mathcal P}$}
  (.45,.833) node {\Large ${\mathcal P}$}
  (2.2,.8) node {\Large $\overline{\mathcal P}'''$}
  (.8,2.2) node {\Large $\mathcal P'''$}
  (1.85,2.2) node {\Large $\overline{\mathcal P}''$}
  (2.2,1.85) node {\Large $\mathcal P''$}
  (.3,1.3) node {\Large $\overline{\mathcal P}'$}
  (1.3,.3) node {\Large $\mathcal P'$}
	;

\end{tikzpicture}
\end{center}
\caption{Canonical regions making the parametrization unique.
\label{fig:domains}}
\end{figure}
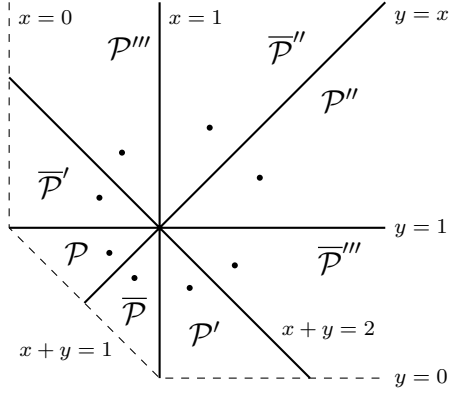
Indeed, the side $(0,1] \times \{1\}$ of $\mathcal P$ represents all trapezoids uniquely. If $Q$ is not a trapezoid and if $P$ is the only parallelogram that is spanned by two sides of $Q$ and contains the fourth vertex of $Q$ in its interior, a suitable affine map sends the common vertices of $Q$ and $P$ consecutively onto $\VECT{0}{1}$, $\VECT{0}{0}$ and $\VECT{1}{0}$, whence the fourth vertex is placed onto some $\VECT{x}{y}$ with $x,y < 1$. If not yet $x \le y$, then exchanging the positions of $\VECT{0}{1}$ and $\VECT{1}{0}$ forces finally $x \le y$. This gives parameters $(x,y)$ within $\mathcal P$. 

Given $(x,y) \in \mathcal P$, one can easily check that the above mentioned (up to) seven other pairs of parameters of $Q$ satisfy
\[
\begin{array}{rcccl}
	(x',y')& \in& \mathcal P'&=&\left\{(x,y) \in \mathbb R^2: y > 0,\, x \ge 1,\, x+y \le 2\right\},\\
	(x'',y'')& \in& \mathcal P''&=&\left\{(x,y) \in \mathbb R^2: y\ge 1,\, x \ge y\right\},\\
  (x''',y''')& \in& \mathcal P'''&=&\left\{(x,y) \in \mathbb R^2: x > 0,\, x \le 1,\, x+y \ge 2\right\},\\
	(\overline{x},\overline{y})& \in& \overline{\mathcal P}&=&\{(x,y) \in \mathbb R^2: x+y > 1,\, x \le 1,\, x \ge y\},\\
	(\overline{x}',\overline{y}')& \in& \overline{\mathcal P}'&=&\left\{(x,y) \in \mathbb R^2: x > 0,\, y \ge 1,\, x+y \le 2\right\},\\
	(\overline{x}'',\overline{y}'')& \in& \overline{\mathcal P}''&=&\left\{(x,y) \in \mathbb R^2: x\ge 1,\, x \le y\right\},\\
	(\overline{x}''',\overline{y}''')& \in& \overline{\mathcal P}'''&=&\left\{(x,y) \in \mathbb R^2: y > 0,\, y \le 1,\, x+y \ge 2\right\}.	
\end{array}
\]

Now every affine type of convex quadrangles has a unique parametrization within each of the eight regions $\mathcal P, \ldots, \overline{\mathcal P}'''$. The eight dots in Figure~\ref{fig:domains} represent the eight pairs of parameters based on $(x,y)=\left(\frac{2}{3},\frac{5}{6}\right)$. The common point $(1,1)$ of all regions gives the unique parametrization of all parallelograms. All other pairs of parameters on the lines $x=1$ and $y=1$ represent trapezoids, each having four parametrizations. All other pairs on the lines $y=x$ and $x+y=2$ represent affine images of kites, also each having four parametrizations. 


\section{$3$-self-affinity based on glass-cut dissections\label{sec:glass-cut}}

A \emph{glass-cut} of a topological disc $D$ splits it into two sub-discs by a cut along a line segment connecting two points of the boundary of $D$. This results in a first \emph{glass-cut dissection} of $D$. Given a glass-cut dissection of $D$, refining it by an additional glass-cut of one of its pieces gives a further glass-cut dissection. That is, a glass-cut dissection of $D$ into discs $D_1,\ldots,D_n$ is obtained by $n-1$ successive glass-cuts, see \cite{richter2024+} for some related references. 

For every $n \ge 2$, the paper \cite{richter2024+} characterizes all topological discs admitting an $n$-self-affinity based on a glass-cut dissection. The last condition is so restrictive, that all such discs are not only convex, but even always triangles or quadrangles \cite[Theorem~1(i)]{richter2024+}. 

If a $3$-self-affinity of a convex quadrangle $Q$ is based on a glass-cut dissection, its combinatorial structure (describing the mutual inclusions of the vertices and the relative open sides of $Q$ and of its three pieces) is one of the two \emph{combinatorial types} $\mathcal A$ or $\mathcal B$ from Figure~\ref{fig:combAB}, as follows from \cite[Lemma~10]{richter2024+}.
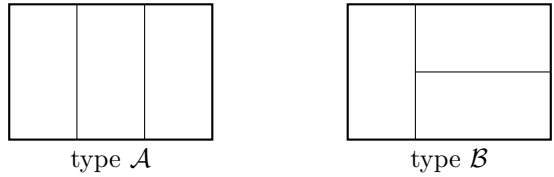
\begin{figure}
\begin{center}
\begin{tikzpicture}[xscale=.9,yscale=.9]

\draw[thick]
  (0,2)--(0,0)--(3,0)--(3,2)--cycle
	(5,2)--(5,0)--(8,0)--(8,2)--cycle
	;

\draw
  (1,2)--(1,0)
	(2,2)--(2,0)
	(6,2)--(6,0)
	(6,1)--(8,1)
	(1.5,0) node[below] {type $\mathcal A$}
	(6.5,0) node[below] {type $\mathcal B$}
	;

\end{tikzpicture}
\end{center}
\caption{The combinatorial types $\mathcal A$ and $\mathcal B$ of glass-cut dissections.\label{fig:combAB}}
\end{figure}

\begin{prop}\label{prop:trapezoidsAB}
Every trapezoid admits $3$-self-affinities under glass-cut dissections of combinatorial type $\mathcal A$ as well as of combinatorial type $\mathcal B$.
\end{prop}

\begin{proof}
The affine type of a trapezoid $T$ is determined by the ratio $z \in (0,1]$ of the lengths of its parallel sides. A $3$-self-affinity of type $\mathcal A$ is obtained by two cuts connecting the parallel sides of $T$ and partitioning them proportionally, see the left-hand part of Figure~\ref{fig:trapezoidsAB}.
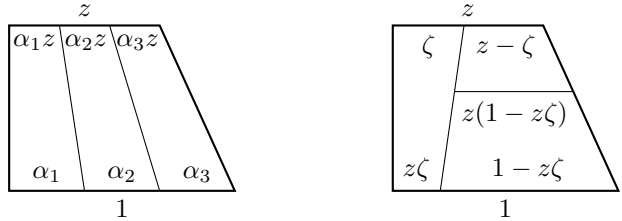
\begin{figure}
\begin{center}
\begin{tikzpicture}[xscale=3,yscale=2.2]

\draw[thick]
  (0,1)--(0,0)--(1,0)--(.6666,1)--cycle
  (-1.7,1)--(-1.7,0)--(-.7,0)--(-1.033,1)--cycle
	;

\draw
  (-1.477,1)--(-1.366,0)
	(-1.255,1)--(-1.033,0)
	(-1.2,0) node[below] {$1$}
	(-1.533,0) node[above] {$\alpha_1$}
	(-1.2,0) node[above] {$\alpha_2$}
	(-.866,0) node[above] {$\alpha_3$}
	(-1.366,1) node[above] {$z$}
	(-1.588,1) node[below] {$\alpha_1 z$}
	(-1.36,1) node[below] {$\alpha_2 z$}
	(-1.13,1) node[below] {$\alpha_3 z$}
	(.5,0) node[below] {$1$}
	(.3333,1) node[above] {$z$}
	(.158,1) node[below] {$\zeta$}
	(.5,1) node[below] {$z-\zeta$}
	(.105,0) node[above] {$z\zeta$}
	(.6,0) node[above] {$1-z\zeta$}
	(.54,.6) node[below] {$z(1-z\zeta)$}
  (.2105,0)--(.3157,1)
	(.2736,.6)--(.8,.6)
	;

\end{tikzpicture}
\end{center}
\caption{$3$-self-affinities of types $\mathcal A$ and $\mathcal B$ of a trapezoid.\label{fig:trapezoidsAB}}
\end{figure}

For type $\mathcal B$, define $\zeta=\frac{z}{1+z+z^2} \in (0,z)$.
The right-hand part of Figure~\ref{fig:trapezoidsAB} illustrates a dissection of $T$ where the lengths of the horizontal sides are indicated. Of course, the left-hand and the right-hand lower piece are of the same affine type as $T$. For verifying the same for the right-hand upper piece, one has to show that $z-\zeta=z(z(1-z\zeta))$. This is a simple consequence of the definition of $\zeta$.
\end{proof}

We come to non-trapezoidal convex quadrangles. In \cite[Section~4]{richter2024+} another pa\-rame\-tri\-za\-tion of the affine types of such quadrangles $Q$ is introduced, which has advantages when studying glass-cuts, see Figure~\ref{fig:gcparameter}.
\begin{figure}
\begin{center}
\begin{tikzpicture}[scale=.7]

\fill
  (0,0) circle (.06) node[left] {\footnotesize $v_2=\VECT{0}{0}$}
	(4,0) circle (.06) node[below] {\footnotesize $v_3=\VECT{1-\alpha}{0}$}
	(6,0) circle (.06) 
	(6.2,0) node[right] {\footnotesize $s=\VECT{1}{0}$}
  (3,1.5) circle (.06) 
	(3,1.5) node[above right] {\footnotesize $v_4=\VECT{1-\beta}{1-\frac{1-\beta}{1-\alpha}}=\VECT{1-\beta}{1-\bar{\beta}}$}
	(0,3) circle (.06) node[left] {\footnotesize $v_1=\VECT{0}{1-\frac{(1-\beta)\alpha}{(1-\alpha)\beta}}=\VECT{0}{1-\bar{\alpha}}$}
	(0,6) circle (.06) node[left] {\footnotesize $t=\VECT{0}{1}$}
	(1.6,.9) node {\Large $\begin{array}{l} Q(\alpha,\beta)\\
	= Q(\bar{\alpha},\bar{\beta})\end{array}$}
	;

\draw[thick]
  (0,0)--(4,0)--(3,1.5)--(0,3)--cycle
  ;
  
\draw
  (0,3)--(0,6)--(3,1.5)--(6,0)--(4,0)
  ;

\end{tikzpicture}
\end{center}
\caption{gc-parametrization of non-trapezoids.\label{fig:gcparameter}}
\end{figure}
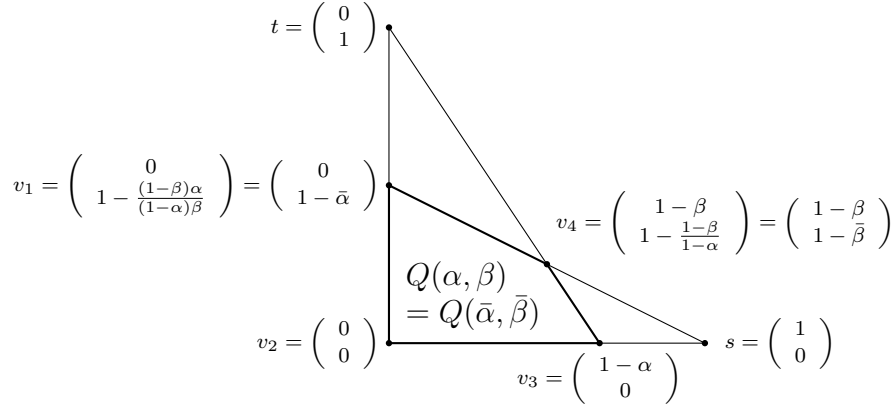
Let $Q$ have the vertices $v_1,v_2,v_3,v_4$. Prolonging opposite sides of $Q$ up to their intersection points erects two triangles over two sides of $Q$; say, we get $\triangle v_3 s v_4$ over the side $v_3v_4$ and $\triangle v_4tv_1$ over $v_4v_1$. Applying a suitable affine transformation, we can assume that $v_2=\VECT{0}{0}$, $s=\VECT{1}{0}$ and $t=\VECT{0}{1}$. Then the first coordinates $1-\alpha$ of $v_3$ and $1-\beta$ of $v_4$ define the parametrization $Q(\alpha,\beta)$ of the affine type of $Q$. We call it the \emph{gc-parametrization}. Allowing all pairs $(\alpha,\beta)$ with $0 < \alpha < \beta < 1$, we obtain gc-parametrizations for all non-trapezoidal convex quadrangles. 
As the roles of the erected triangles (or of the points $s$ and $t$, respectively) are exchangeable, the parameters are not unique in general. We obtain 
\[
Q(\alpha,\beta)=Q(\bar{\alpha},\bar{\beta}),
\quad\text{where}\quad (\bar{\alpha},\bar{\beta})=\left(\frac{(1-\beta)\alpha}{(1-\alpha)\beta},\frac{1-\beta}{1-\alpha}\right).
\]
The translation into our natural parametrization is as follows.

\begin{lem}\label{lem:translation}
\begin{itemize}
\item[(i)]
Let $0 < \alpha < \beta < 1$. Then 
$Q(\alpha,\beta)=Q[x,y]$ where
\begin{equation}\label{eq:translation}
(x,y)=\left(\frac{1-\beta}{1-\alpha},\beta\right), \quad\text{i.e.}\quad (\alpha,\beta)=\left(\frac{x+y-1}{x},y\right).
\end{equation}
Under the substitution \eqref{eq:translation}, the restrictions $0 < \alpha < \beta < 1$ on $(\alpha,\beta)$ are equivalent to $x+y > 1$, $x < 1$ and $y<1$ on $(x,y)$ (whence in particular $(x,y) \in \mathcal P \cup \overline{\mathcal P}$).
\item[(ii)]
Moreover, under the suppositions of (i), $\beta>\frac{1}{2-\alpha}$
if and only if $x <y$ (whence in particular $(x,y) \in \mathcal P$).
\end{itemize}
\end{lem}

\begin{proof}
(i): Applying the affine map $\varphi\VECT{\xi}{\eta}= \left(\begin{array}{cc}
\frac{1}{1-\alpha} & 0\\
0 & \frac{(1-\alpha)\beta}{\beta-\alpha}
\end{array}\right)\VECT{\xi}{\eta}$ to the representative of $Q(\alpha,\beta)$ from Figure~\ref{fig:gcparameter}, we obtain
\[
\varphi(v_1)=\VECT{0}{1}, \quad \varphi(v_2)=\VECT{0}{0}, \quad \varphi(v_3)=\VECT{1}{0}, \quad \varphi(v_4)=\VECT{\frac{1-\beta}{1-\alpha}}{\beta}.
\]
Hence we got a canonical representative of the natural parametrization as in Figure~\ref{fig:Q[x,y]}, whence we obtain \eqref{eq:translation}. Moreover, one easily checks that the inequalities $0 < \alpha < \beta < 1$ are equivalent to $x+y > 1$, $x< 1$ and $y<1$.

(ii): Under the conditions $x+y>1$, $x<1$ and $y<1$, the inequality $\beta > \frac{1}{2-\alpha}$ is equivalent to $x < y$, as can easily be calculated. 
\end{proof}

The characterization of all convex quadrangles admitting $3$-self-affinities under glass-cut dissections is given in terms of the gc-parametrization in \cite[Theorem~18]{richter2024+}. Here we cite only the cases of non-trapezoids. The attribution to the combinatorial types from Figure~\ref{fig:combAB} is not explicitly given in Theorem~18 from \cite{richter2024+}, but can easily be seen from the proof given there.

\begin{prop}[{cf. \cite[Theorem~18]{richter2024+}, see also \cite[Satz~4]{hertel2000} for the first family and \cite[Sections 5 and 6]{richter2010} for the second and the third family}]\,

\begin{itemize}
\item[(i)]
A non-trapezoidal convex quadrangle admits a $3$-self-affinity based on a glass-cut dissection of combinatorial type $\mathcal A$ if and only if it is of an affine type $Q(\alpha,\beta)$ with $(\alpha,\beta)$ from
\[
\mathbf A^\mathrm{gc}=\left\{(\alpha,\beta):0 < \alpha < \beta < 1,\,\beta=\frac{-1+\sqrt{1+4\alpha-4\alpha^2}}{2\alpha(1-\alpha)}\right\}.
\]
\item[(ii)]
A non-trapezoidal convex quadrangle admits a $3$-self-affinity based on a glass-cut dissection of combinatorial type $\mathcal B$ if and only if it is of an affine type $Q(\alpha,\beta)$ with $(\alpha,\beta)$ from one of the two families
\[
\mathbf B^\mathrm{gc}_1=\left\{(\alpha,\beta): 0 < \alpha < \beta < 1,\, \beta=\frac{1-3\alpha+\alpha^2+\sqrt{1-2\alpha+7\alpha^2-6\alpha^3+\alpha^4}}{2(1-\alpha)}\right\}
\]
or
\begin{align*}
\mathbf B^\mathrm{gc}_2=\Big\{&(\alpha,\beta): 
0 < \alpha < \beta < 1,\\
&\left(\alpha-\alpha^2\right)\beta^3+\left(1-2\alpha+2\alpha^2\right)\beta^2+\left(-1+2\alpha-4\alpha^2+\alpha^3\right)\beta+\alpha^2=0 \Big\}.
\end{align*}
\end{itemize}
\end{prop}

We translate the last result into the natural parametrization, picking unique parameters from the region $\mathcal P=\{(x,y): x+y >1,\, y \le 1,\, x \le y\}$.

\begin{thm}\label{thm:AB}
\begin{itemize}
\item[(i)]
A non-trapezoidal convex quadrangle admits a $3$-self-affinity based on a glass-cut dissection of combinatorial type $\mathcal A$ if and only if it is of an affine type $Q[x,y]$ with $(x,y)$ from
\[
\mathbf A=\left\{(x,y): x+y > 1,\, y<1,\, x< y,\, y^3+xy^2-x^2-y^2=0\right\}.
\]
\item[(ii)]
A non-trapezoidal convex quadrangle admits a $3$-self-affinity based on a glass-cut dissection of combinatorial type $\mathcal B$ if and only if it is of an affine type $Q[x,y]$ with $(x,y)$ from one of the two families
\[
\mathbf B_1=\left\{(x,y): x+y > 1,\, y<1,\, x<y,\,(x+1)y^2-(x+1)y+x(1-x)=0\right\}
\]
or
\begin{align*}
\mathbf B_2=\big\{(x,y):\,& x+y > 1,\, y<1,\, x<y,\,
\\
& x^{3}+\left(-y^{2}+y-2\right) x^{2}+\left(-y^{3}+2 y^{2}-y+1\right) x+y^{2}-y=0\big\}.
\end{align*}
\end{itemize}
\end{thm}

\begin{proof}
We know from Lemma~\ref{lem:translation}(i) that the translation of parametrizations is based on the substitutions $\alpha=\frac{x+y-1}{x}$ and $\beta=y$ and that the conditions $0 < \alpha < \beta < 1$ are equivalent to $x+y > 1$, $x<1$ and $y<1$, so that we obtain $(x,y) \in \mathcal P \cup \overline{\mathcal P}$. We have to show that, having supposed that the last conditions are satisfied, the characterizing equations concerning $(\alpha,\beta)$ from $\mathbf A^\mathrm{gc}$, $\mathbf B_1^\mathrm{gc}$ and $\mathbf B_2^\mathrm{gc}$ are equivalent to those on $(x,y)$ in $\mathbf A$, $\mathbf B_1$ and $\mathbf B_2$, respectively. Moreover, one needs to show that all solutions $(x,y)$ of these equations subject to $x+y > 1$, $x<1$ and $y<1$ satisfy $x < y$.

The family $\mathbf A^\mathrm{gc}$: We multiply the equation $\beta=\frac{-1+\sqrt{1+4\alpha-4\alpha^2}}{2\alpha(1-\alpha)}$ by the positive factor $2\alpha(1-\alpha)$, add $1$, observe that the left-hand side $2\alpha(1-\alpha)\beta+1$ is positive, square the equation, subtract $1+4\alpha-4\alpha^2$ and divide by the positive term $4\alpha(1-\alpha)$, this way arriving at
\[
\alpha(1-\alpha)\beta^2+\beta-1=0.
\]
Next we substitute $\alpha=\frac{x+y-1}{x}$ and $\beta=y$, multiply by the positive factor $x^2$ and divide by the positive term $1-y$, which yields
\begin{equation}\label{eq:eqA}
y^3+xy^2-x^2-y^2=0.
\end{equation}
So the parameters $(\alpha,\beta)$ from $\mathbf A^\mathrm{gc}$ correspond to the parameters $(x,y)$ from
\[
\mathbf A=\left\{(x,y):x+y>1,\,x<1,\, y<1,\,y^3+xy^2-x^2-y^2=0\right\}.
\]

To see that every $(x,y) \in \mathbf A$ satisfies $x < y$, note that \eqref{eq:eqA} is equivalent to $(x+y-1)y^2=x^2$, whence $y^2 > x^2$, since $0< x+y-1 <1$, and finally $y>x$, as $x,y >0$.

The family $\mathbf B_1^\mathrm{gc}$: We multiply the equation $\beta=\frac{1-3\alpha+\alpha^2+\sqrt{1-2\alpha+7\alpha^2-6\alpha^3+\alpha^4}}{2(1-\alpha)}$ by the positive factor $2(1-\alpha)$ and subtract the right-hand side, which yields
\begin{equation}\label{eq:B1eq1}
2(1-\alpha)\beta+\left(-1+3\alpha-\alpha^2\right)-\sqrt{(-1+3\alpha-\alpha^2)^2+4\alpha(1-\alpha)}=0.
\end{equation}
Next note that the term
\[
2(1-\alpha)\beta+\left(\left(-1+3\alpha-\alpha^2\right)+\sqrt{(-1+3\alpha+\alpha^2)^2+4\alpha(1-\alpha)}\right)
\]
is positive. We multiply \eqref{eq:B1eq1} with this very term and divide by the negative term $-4(1-\alpha)$, obtaining
\[
-(1-\alpha)\beta^2+\left(1-3\alpha+\alpha^2\right)\beta+\alpha=0.
\]
Now we substitute $\alpha=\frac{x+y-1}{x}$ and $\beta=y$, multiply by the positive factor $x^2$, divide by the negative term $y-1$ and arrive at
\begin{equation}\label{eq:B1eq2}
(x+1)y^2-(x+1)y+x(1-x)=0.
\end{equation}
The parameters $(\alpha,\beta)$ from $\mathbf B_1^\mathrm{gc}$ are translated into $(x,y)$ from
\[
\mathbf B_1=\left\{(x,y):x+y>1,\,x<1,\, y<1,\,(x+1)y^2-(x+1)y+x(1-x)=0\right\}.
\]

Finally, to see that $x < y$ for all $(x,y) \in \mathbf B_1$, we add $xy(1-y)$ to \eqref{eq:B1eq2} and divide by the negative term $1-x-y$, which gives
\[
x-y=-\frac{xy(1-y)}{x+y-1}<0.
\]

The family $\mathbf B_2^\mathrm{gc}$: Now we have
\[
\left(\alpha-\alpha^2\right)\beta^3+\left(1-2\alpha+2\alpha^2\right)\beta^2+\left(-1+2\alpha-4\alpha^2+\alpha^3\right)\beta+\alpha^2=0.
\]
Substituting $\alpha=\frac{x+y-1}{x}$ and $\beta=y$ and multiplying with the positive factor $\frac{x^3}{(1-y)^2}$ we get
\begin{align*}
\mathbf B_2=\big\{(x,y):\,&x+y>1,\,x<1,\, y<1,\,\\
&x^{3}+\left(-y^{2}+y-2\right) x^{2}+\left(-y^{3}+2 y^{2}-y+1\right) x+y^{2}-y=0\big\}.
\end{align*}

For showing that every $(x,y) \in \mathbf B_2$ satisfies $x<y$, we apply Lemma~\ref{lem:translation}(ii). That is, we have to show that $\beta>\frac{1}{2-\alpha}$ for every fixed $(\alpha,\beta) \in \mathbf B_2^\mathrm{gc}$. As $(\alpha,\beta) \in \mathbf B_2^\mathrm{gc}$, we have $0 < \alpha <\beta < 1$ and $\beta$ is a zero of the polynomial 
\[
p_\alpha(\xi)=\left(\alpha-\alpha^2\right)\xi^3+\left(1-2\alpha+2\alpha^2\right)\xi^2+\left(-1+2\alpha-4\alpha^2+\alpha^3\right)\xi+\alpha^2.
\]
The restriction $0 < \alpha < 1$ yields 
\[
\alpha < \frac{1}{2-\alpha} < 1,
\]
and we have
\[
p_\alpha(\alpha)=-\alpha(1-\alpha)^4<0, \; p_\alpha\left(\frac{1}{2-\alpha}\right)=-\frac{2(1-\alpha)^4}{(2-\alpha)^3}<0, \; p_\alpha(1)=\alpha(1-\alpha)^2>0.
\]
Moreover, the polynomial $p_\alpha$ is convex on $[\alpha,1]$, since
\[
p_\alpha''(\xi)=6\alpha(1-\alpha)\xi+2(1-\alpha)^2+2\alpha^2 >0
\]
for $\xi \ge 0$. Thus the only zero $\beta$ of $p_\alpha$ in $(\alpha,1)$ is in $\left(\frac{1}{2-\alpha},1\right)$, whence $\beta > \frac{1}{2-\alpha}$.
\end{proof}


\section{$3$-self-affinity based on non-glass-cut dissections\label{sec:non-glass-cut}}

If a convex quadrangle $Q$ is dissected into three convex quadrangles $Q_1$, $Q_2$ and $Q_3$, but not based on a glass-cut dissection, every dissecting line segment starting from the boundary of $Q$ must end in the interior of $Q$. Then, by convexity, this point is the end of three dissecting line segments and is a common point of $Q_1$, $Q_2$ and $Q_3$. Hence the combinatorics of the dissection is necessarily as in the left-hand part of Figure~\ref{fig:combC}. We call it \emph{combinatorial type $\mathcal C$}.
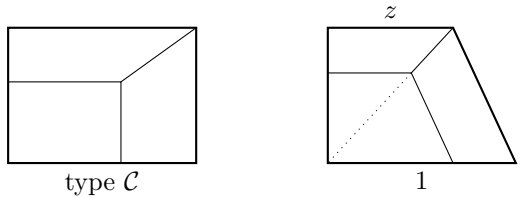
\begin{figure}
\begin{center}
\begin{tikzpicture}[xscale=2.5,yscale=1.8]

\draw[thick]
  (0,1)--(0,0)--(1,0)--(.666,1)--cycle
  (-1.7,1)--(-1.7,0)--(-.7,0)--(-.7,1)--cycle
	;

\draw
  (-1.7,.6)--(-1.1,.6)--(-1.1,0)
	(-1.1,.6)--(-.7,1)
	(-1.2,0) node[below] {type $\mathcal C$}
	(0,.666)--(.444,.666)--(.666,0)
	(.444,.666)--(.666,1)
	(.5,0) node[below] {$1$}
	(.3333,1) node[above] {$z$}
	;

\draw[dotted]
  (.444,.666)--(0,0)
	;

\end{tikzpicture}
\end{center}
\caption{The last combinatorial type $\mathcal C$ and a corresponding $3$-self-affinity of a trapezoid.\label{fig:combC}}
\end{figure}

\begin{prop}\label{prop:trapezoidsC}
A trapezoid admits a $3$-self-affinity of combinatorial type $\mathcal C$ if and only if it is not a parallelogram.
\end{prop}

\begin{proof}
If a parallelogram $P$ had a $3$-self-affinity of type $\mathcal C$, then each of the three sides of pieces emanating from their common vertex in the interior of $P$ had to be parallel to one of the sides of $P$. Hence two of these sides had the same direction, contradicting the condition of type $\mathcal C$.

The right-hand part of Figure~\ref{fig:combC} illustrates a $3$-self-affinity of a trapezoid $T$ whose quotient of the lengths of its parallel sides is $z \in (0,1)$. The left-hand lower piece is obtained by dilating $T$ with factor $z$ and with respect to its left-hand lower vertex. Then it is clear that the respective ratios of the lengths of the parallel sides of the two other pieces are $z$ as well.
\end{proof}

\begin{thm}\label{thm:C}
A non-trapezoidal convex quadrangle admits a $3$-self-affinity of combinatorial type $\mathcal C$ if and only if it is of an affine type $Q[x,y]$ with $(x,y)$ either from the family
\[
\mathbf C=\left\{(x,x^2-x+1): 0 < x < 1\right\}
\]
or with $(x,y)$ being one of the $13$ pairs from Table~\ref{tab:1}. All above parameters are in the region $\mathcal P$.
\end{thm}

\begin{table}[t]
\begin{center}
\begin{tabular}{|c|c|}
\hline
no. & expression of parameter $(x,y) \in \mathcal P$, implicit\\
& (by a system of two equations) and explicit (first digits)\\
\hline\hline
 & $x^{3} y-x y+y^{2}-x-y+1=0,$ \\
1 & $x^{2} y^{2}+x^{3}-y^{2} x-2 x^{2}-x y-y^{2}+2 x+2 y-1=0,$ \\
\cline{2-2} 
& $(0.54368..., 0.83928...)$\\ 
\hline
& $x^{3} y+x^{2} y^{2}-x^{2} y-y^{2} x-x+y=0,$ \\
2 & $x^{3} y+x^{2} y^{2}-x^{2} y-2 y^{2} x-y^{3}-x^{2}+x y+2 y^{2}+x-y=0$ \\
\cline{2-2}& $(0.55706...,0.85490...)$\\
\hline
& $x^{3} y+x^{2} y^{2}-x^{2} y+y^{2} x-x^{2}-2 x y-y^{2}+x+y=0,$\\
3 & $x^{3} y+x^{2} y^{2}-x^{3}-4 x^{2} y-y^{2} x+2 x^{2}+3 x y+y^{2}-x-y=0$\\
\cline{2-2}& $(0.54660...,0.72669...)$\\
\hline
& $x^{2} y^{2}+y^{3} x+x^{2} y-y^{2} x-x^{2}-2 x y-y^{2}+x+y=0,$\\
4& $x^{4}+x^{3} y-3 x^{3}-2 x^{2} y-y^{2} x+2 x^{2}+3 x y+y^{2}-x-y=0$\\
\cline{2-2}& $(0.50678...,0.67567...)$\\
\hline
& $x^{3} y+2 x^{2} y^{2}+y^{3} x-2 x^{2} y-3 y^{2} x-y^{3}-x^{2}+2 x y+y^{2}=0,$\\
5& $x^{3} y+2 x^{2} y^{2}+y^{3} x-x^{3}-3 x^{2} y-3 y^{2} x-y^{3}+2 x^{2}+x y+y^{2}=0$\\
\cline{2-2}& $\left(\frac{9-4\sqrt{2}}{7},\frac{10+\sqrt{2}}{14}\right)=(0.47759...,0.81530...)$\\
\hline
& $x^{2} y^{2}+y^{3} x-y^{2} x-y^{3}-x+y=0,$\\
6& $x^{2} y^{2}+y^{3} x+x^{3}-x^{2} y-3 y^{2} x-y^{3}-x^{2}+x y+2 y^{2}+x-y=0$\\
\cline{2-2}& $(0.25805...,0.84781...)$\\
\hline
& $x^{3}+x^{2} y-x^{2}+y^{2}-2 x-2 y+2=0,$\\
7& $y^{2} x+y^{3}-x^{2}-3 x y-2 y^{2}+3 x+3 y-2=0$\\
\cline{2-2}& $(0.58750...,0.78257...)$\\
\hline
& $x^{2} y+y^{2} x-2 x^{2}-3 x y-y^{2}+3 x+3 y-2=0,$\\
8& $x^{2} y+y^{2} x-2 x-2 y+2=0$\\
\cline{2-2}& $\left(\frac{1}{2},\frac{7-\sqrt{17}}{4}\right)=(0.5,0.71922...)$\\
\hline
& $x^{2} y+2 y^{2} x+y^{3}-4 x y-4 y^{2}+x+3 y=0,$\\
9& $x^2 + xy - y=0$\\
\cline{2-2}& $(0.59100...,0.85403...)$\\
\hline
& $x^{2} y+y^{3}-2 x y-2 y^{2}+x+y=0,$\\
10& $x^{3}+x^{2} y+2 y^{2} x-2 x^{2}-3 x y-y^{2}+x+y=0$\\
\cline{2-2}& $(0.41803...,0.71831...)$\\
\hline
& $y^{3}x+x^{2} y-y^{3}-x y+y^{2}-x-y+1=0,$\\
11& $x^{3} y-x^{2} y^{2}- y^{3}x-2 x^{3}+2 y^{2} x+2 x^{2}+x y+y^{2}-2 x-2 y+1=0$\\
\cline{2-2}& $(0.33133...,0.78783...)$\\
\hline
& $2 y^{2} x-2 y x-2 y^{2}+x+y=0,$\\
12& $3 x^{2} y+y^{2} x-2 x^{2}-3 x y-y^{2}+x+y=0$\\
\cline{2-2}& $\left(\frac{2}{5},\frac{2}{3}\right)=(0.4,0.66666...)$\\
\hline
& $y^{4}+x^{2} y-y^{3}-x y+y^{2}-x-y+1=0,$\\
13& $x^{4}-x^{2} y^{2}-y^{3} x-2 x^{3}+2 y^{2} x+2 x^{2}+x y+y^{2}-2 x-2 y+1=0$\\
\cline{2-2}& $(0.59717...,0.87586...)$\\
\hline
\end{tabular}
\end{center}
\caption{All singular solutions $(x,y) \in \mathcal P$.\label{tab:1}}
\end{table}

\begin{proof}
We describe the question analytically. The quadrangle $Q=q_1q_2q_3q_4$ is dissected into $A=a_1a_2a_3a_4$, $B=b_1b_2b_3b_4$ and $C=c_1c_2c_3c_4$ as in Figure~\ref{fig:Canalytic}.
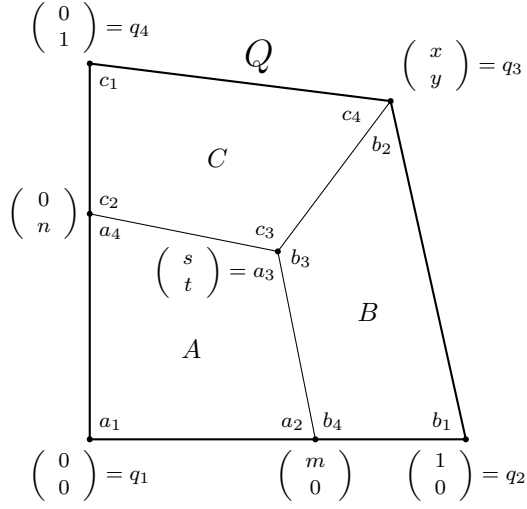
\begin{figure}
\begin{center}
\begin{tikzpicture}[xscale=5,yscale=5]

\draw[thick]
  (0,1)--(0,0)--(1,0)--(.8,.9)--cycle
	;
	
\draw
  (.6,0)--(.5,.5)--(0,.6)
	(.5,.5)--(.8,.9)
	;

\fill
  (0,0) circle (.008) node[below] {\footnotesize $\VECT{0}{0}=q_1$}
	(0,0) node[above right] {\footnotesize $a_1$}
	(.6,0) circle (.008) node[below] {\footnotesize $\VECT{m}{0}$}
	(.595,0) node[above left] {\footnotesize $a_2$}
	(.595,0) node[above right] {\footnotesize $b_4$}
  (1,0) circle (.008) node[below] {\footnotesize $\VECT{1}{0}=q_2$} 
	(.99,0) node[above left] {\footnotesize $b_1$}
  (.8,.9) circle (.008) node[above right] {\footnotesize $\VECT{x}{y}=q_3$}
	(.83,.83) node[below left] {\footnotesize $b_2$}
	(.75,.9) node[below left] {\footnotesize $c_4$}
  (0,1) circle (.008) node[above] {\footnotesize $\VECT{0}{1}=q_4$} 
  (0,.99) node[below right] {\footnotesize $c_1$} 
  (0,.6) circle (.008) node[left] {\footnotesize $\VECT{0}{n}$} 
  (0,.59) node[above right] {\footnotesize $c_2$} 
  (0,.59) node[below right] {\footnotesize $a_4$} 
  (.5,.5) circle (.008)
	(.52,.54) node[below left] {\footnotesize $\VECT{s}{t}=a_3$} 
	(.52,.51) node[above left] {\footnotesize $c_3$} 
	(.51,.53) node[below right] {\footnotesize $b_3$} 
	(.45,.95) node[above] {\LARGE $Q$}
	(.27,.24) node {$A$}
	(.74,.34) node {$B$}
	(.34,.75) node {$C$}
	;

\end{tikzpicture}
\end{center}
\caption{Analytic modelling of type $\mathcal C$.\label{fig:Canalytic}}
\end{figure}
Let
\[
q_1=a_1=\VECT{0}{0},\;q_2=b_1=\VECT{1}{0},\;q_3=b_2=c_4=\VECT{x}{y},\;q_4=c_1=\VECT{0}{1},
\]
\[
a_2=b_4=\VECT{m}{0},\;a_4=c_2=\VECT{0}{n},\;a_3=b_3=c_3=\VECT{s}{t}.
\]
Since $A$, $B$ and $C$ are affine images of $Q$, there exist permutations
\[
\pi_A,\pi_B,\pi_C \in \{(1234),(2341),(3412),(4123),(1432),(2143),(3214),(4321)\}
\]
($\pi=(ijkl)$ standing for $\pi(1)=i$, $\pi(2)=j$, $\pi(3)=k$ and $\pi(4)=l$)
and affine maps
\begin{align*}
\alpha(\cdot)&=
\left(\begin{array}{cc}
\alpha_{11} & \alpha_{12} \\
\alpha_{21} & \alpha_{22}
\end{array}\right)(\cdot)+\VECT{\alpha_{13}}{\alpha_{23}},\\
\beta(\cdot)&=
\left(\begin{array}{cc}
\beta_{11} & \beta_{12} \\
\beta_{21} & \beta_{22}
\end{array}\right)(\cdot)+\VECT{\beta_{13}}{\beta_{23}},\\
\gamma(\cdot)&=
\left(\begin{array}{cc}
\gamma_{11} & \gamma_{12} \\
\gamma_{21} & \gamma_{22}
\end{array}\right)(\cdot)+\VECT{\gamma_{13}}{\gamma_{23}}
\end{align*}
such that
\begin{equation}\label{eq:system}
\alpha(q_i)=a_{\pi_A(i)},\;
\beta(q_i)=b_{\pi_B(i)},\;
\gamma(q_i)=c_{\pi_C(i)} \quad\text{for}\quad i=1,2,3,4.
\end{equation}
Then \eqref{eq:system} is a system of $24$ algebraic equations in the $24$ real unknowns $\alpha_{jk}$, $\beta_{jk}$, $\gamma_{jk}$, $j=1,2$, $k=1,2,3$ and $x,y,s,t,m,n$.
As $Q$ must be a convex quadrangle and $\VECT{m}{0}$ and $\VECT{0}{n}$ have to be between $q_1$ and $q_2$ and between $q_1$ and $q_4$, respectively, we have the restrictions
\begin{equation}\label{eq:restrictions1}
x>0,\; y>0,\; x+y > 1,\; 0 < m < 1, \; 0 < n < 1.
\end{equation}
The system \eqref{eq:system} ensures in particular that $a_1$ and $a_3$ are on different sides of the segment $a_2a_4$, and the same applies to $b_1$ and $b_3$ as well as to $c_1$ and $c_3$ with respect to $b_2b_4$ and $c_2c_4$, respectively. Hence $\VECT{s}{t}$ is in the interior of $Q$ as soon as \eqref{eq:system} and \eqref{eq:restrictions1} are satisfied.

We see that, for every fixed choice of $\pi_A,\pi_B,\pi_C$, every solution of \eqref{eq:system} with \eqref{eq:restrictions1} corresponds to a $3$-self-affinity of combinatorial type $\mathcal C$ of some quadrangle of affine type $Q[x,y]$. Conversely, every such $3$-self-affinity can be described by corresponding permutations $\pi_A,\pi_B,\pi_C$ and a system \eqref{eq:system} with restrictions \eqref{eq:restrictions1}. Moreover, as a reflection of such a dissection with respect to the axis $x=y$ does not affect the combinatorial type $\mathcal C$, we can assume the additional restriction
\begin{equation}\label{eq:restrictions2}
x \le y
\end{equation}
without loss of generality.

Now we have to solve all restricted systems \eqref{eq:system}, \eqref{eq:restrictions1}, \eqref{eq:restrictions2}, one for each of the $512$ possible choices of $(\pi_A,\pi_B,\pi_C)$, for the unknowns $(\alpha_{11},\ldots,\gamma_{23},x,y,s,t,m,n)$, this way finding a parametrization $Q[x,y]$ of every quadrangle admitting a $3$-self-affinity of type $\mathcal C$.

This has been done systematically by the aid of the computer algebra system Maple, see \cite{maple} for details. Only few choices $(\pi_A,\pi_B,\pi_C)$ give rise to solutions. As we are not interested in trapezoidal solutions, we collected only those permutations leading to solutions with $x \ne 1$ and $y \ne 1$. These are
\begin{eqnarray*}
&\pi^f_1=(1432,1234,4321),\; \pi^f_2=(1432,3214,2341),\; \pi^f_3=(3214,1234,1432),&\\
&\pi^f_4=(3214,4123,2143),\; \pi^f_5=(3214,2143,4123),\; \pi^f_6=(3214,3214,3412),&
\end{eqnarray*}
each giving an uncountable family of solutions, and 
\begin{eqnarray*}
&\pi^s_1=(2341,1234,1234),\;\pi^s_2=(2341,1234,2143),\;\pi^s_3=(2341,2341,1234),&\\
&\pi^s_4=(2341,2341,1432),\;\pi^s_5=(2341,2341,2143),\;\pi^s_6=(2341,1432,2143),&\\
&\pi^s_7=(2341,2143,1234),\;\pi^s_8=(2341,2143,1432),\;\pi^s_9=(2341,2143,2143),&\\
&\pi^s_{10}=(4123,1234,4321),\;\pi^s_{11}=(4123,3214,1432),\;\pi^s_{12}=(4123,3214,4321),&\\
&\pi^s_{13}=(2143,1234,3412),\;\pi^s_{14}=(2143,2341,1432),\;\pi^s_{15}=(2143,2143,3412)\,&
\end{eqnarray*}
each giving only one single solution.

Solving the system based on $\pi^f_1$ gave solutions only with $0 < x < 1$. Then all other unknowns had a unique expression in terms of $x$. In particular, $y=x^2-x+1$, so that we obtained the parameters $(x,y)=(x,x^2-x+1) \in \mathcal P$, $0 < x < 1$. This is the family $\mathbf C$ of solutions mentioned in Theorem~\ref{thm:C}.

For $\pi^f_2$, we got the family of solutions $(x,x^2-x+1)$ with $x > 0$. These are in $\overline{\mathcal P}''$. Transforming them by the substitution \eqref{eq:sub4} into solutions from $\mathcal P$ results in the same family as we got for $\pi^f_1$.

For $\pi^f_3$, we got the pairs $\left(x,\frac{1}{x}\right)$, $0 < x < 1$, that are in $\mathcal P'''$. Now the substitution \eqref{eq:sub2} translates them into the family that we obtained for $\pi^f_1$.

For $\pi^f_4$, we got $\left(\frac{m^2}{m^2-m+1},\frac{1}{m^2-m+1}\right)$, $0 < m < 1$, from $\overline{\mathcal P}'$. The substitution \eqref{eq:sub3} transforms them into the family that we got for $\pi^f_1$.

For $\pi^f_5$, we got the same pairs $(x,y)$ as we got for $\pi^f_4$. For $\pi^f_6$, we got the same pairs as we got for $\pi^f_3$. So all one-parameter families based on $\pi^f_1,\ldots,\pi^f_6$ agree.

Solving the $15$ systems based on $\pi^s_1,\ldots,\pi^s_{15}$ gives unique solutions; in particular, we get a unique pair $(x,y)$ for each case. For $\pi^s_1,\ldots,\pi^s_9$, these solutions are already in $\mathcal P$. Eliminating the unknowns $\alpha_{11},\ldots,\gamma_{23},s,t,m,n$ from a system gave two polynomial equations in $x$ and $y$ that describe the unique solution $(x,y)$ exactly. For the permutations $\pi^s_7,\pi^s_8,\pi^s_9$, one or both of these polynomial equations could even be simplified by dividing by an irrelevant factor $x$, $y$ or $x+y-2$. The results are presented in rows no. $1,\ldots,9$ of Table~\ref{tab:1}.

The unique solutions for $\pi^s_{10}$, $\pi^s_{12}$ and $\pi^s_{13}$ were found in $\overline{\mathcal P}'$. In a first step we eliminated $\alpha_{11},\ldots,\gamma_{23},s,t,m,n$ to get a system of two equations describing the unique solution in $\overline{\mathcal P}'$. Then we used the substitution \eqref{eq:sub3} for getting two equations describing the unique solution in $\mathcal P$, and we computed the solution numerically. For $\pi^s_{11}$ the first solution was in ${\mathcal P}'''$ and we worked analogously with substitution \eqref{eq:sub2}. These results are given in rows $10,...,13$ of Table~\ref{tab:1}.

For $\pi^s_{14}$, we got the pair $(x,y)=(0.43015\ldots,0.75487\ldots)\in \mathcal P$ and, after dividing by irrelevant factors $x+y$, $x-1$ and $y-1$, the respective system of equations $\{y^2+x-1=0,x^2-x-y+1=0\}$. The second equation shows that this solution $(x,y)$ is not new, as it is part of the family found by $\pi_1^f,\ldots,\pi_6^f$.

Finally, $\pi^s_{15}$ lead to $(x,y)=(0.24512\ldots,1.32471\ldots) \in \overline{\mathcal P}'$. Eliminating the other unknowns, translating the solution into $\mathcal P$ by the substitution \eqref{eq:sub3} and dividing out irrelevant factors $x+y-1$ and $x-y$, we got $(0.43015\ldots,0.75487\ldots)\in \mathcal P$ expressed as the solution of the system $\{x^2-x-y+1=0,x^2y+y^3-xy-y^2-x+y=0\}$. The first equation shows that this solution belongs to the family based on $\pi_1^f,\ldots,\pi_6^f$, so it is not new. In fact, a symbolic calculation shows that this solution agrees with the one obtained from $\pi^s_{14}$.
\end{proof}

Looking back to the proof, we see that the affine type $Q[0.43015...,0.75487...]$ is particularly remarkable, since it admits $3$-self-affinities coming from eight permutations, namely $\pi^s_{14},\pi^s_{15}$ and $\pi^f_1,\ldots,\pi^f_6$. No two of them are the same under affine congruence.
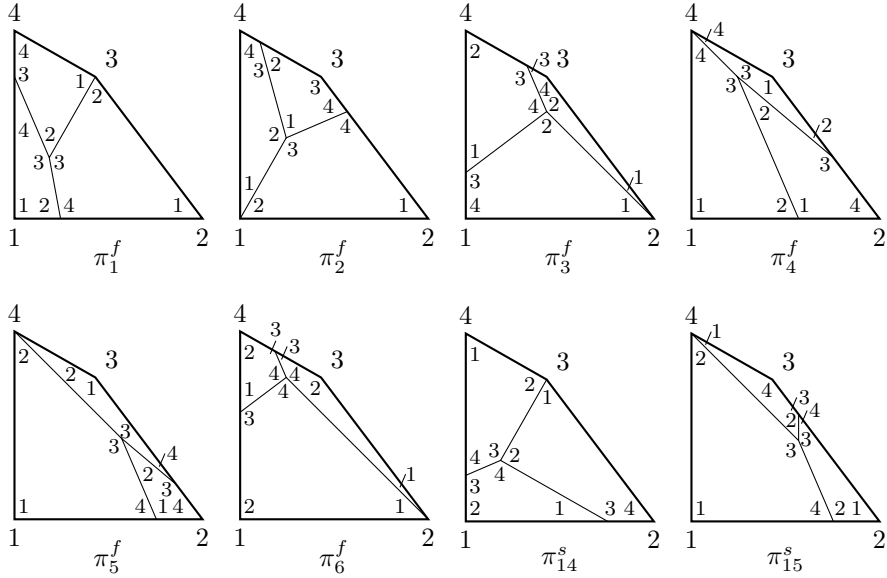
\begin{figure}
\begin{center}
\begin{tikzpicture}

\draw
  (0,0) node {\tikz[scale=2.5]{
	
\draw[thick]
  (0,0)--(1,0)--(.4301,.7548)--(0,1)--cycle
	(0,0) node[below] {1}
	(1,0) node[below] {2}
	(.4301,.7548) node[above right] {3}
	(0,1) node[above] {4}
	;
	
\draw
  (.5,-.2) node {$\pi_1^f$}
  (.2451,0)--(.1850,.3247)--(0,.7548)
	(.1850,.3247)--(.4301,.7548)
	(0.05,0.07) node {\footnotesize 1}
	(0.16,0.07) node {\footnotesize 2}
	(0.29,0.07) node {\footnotesize 4}
	(0.87,0.07) node {\footnotesize 1}
	(0.05,0.47) node {\footnotesize 4}
	(0.05,0.77) node {\footnotesize 3}
	(0.05,0.89) node {\footnotesize 4}
	(0.13,0.3) node {\footnotesize 3}
	(0.24,0.3) node {\footnotesize 3}
	(0.19,0.45) node {\footnotesize 2}
	(0.35,0.73) node {\footnotesize 1}
	(0.44,0.65) node {\footnotesize 2}
	;
  }}

  (3,0) node {\tikz[scale=2.5]{

\draw[thick]
  (0,0)--(1,0)--(.4301,.7548)--(0,1)--cycle
	(0,0) node[below] {1}
	(1,0) node[below] {2}
	(.4301,.7548) node[above right] {3}
	(0,1) node[above] {4}
	;
	
\draw
  (.5,-.2) node {$\pi_2^f$}

  (.5698,.5698)--(.2451,.4301)--(.1054,.9399)
	(.2451,.4301)--(0,0)

	(0.05,0.19) node {\footnotesize 1}
	(0.18,0.45) node {\footnotesize 2}
	(0.1,0.79) node {\footnotesize 3}
	(0.05,0.89) node {\footnotesize 4}

	(0.87,0.07) node {\footnotesize 1}
	(0.1,0.07) node {\footnotesize 2}
	(0.28,0.37) node {\footnotesize 3}
	(0.56,0.49) node {\footnotesize 4}

	(0.27,0.52) node {\footnotesize 1}
	(0.19,0.82) node {\footnotesize 2}
	(0.4,0.7) node {\footnotesize 3}
	(0.47,0.6) node {\footnotesize 4}
	;
  }}

  (6,0) node {\tikz[scale=2.5]{

\draw[thick]
  (0,0)--(1,0)--(.4301,.7548)--(0,1)--cycle
	(0,0) node[below] {1}
	(1,0) node[below] {2}
	(.4301,.7548) node[above right] {3}
	(0,1) node[above] {4}
	;
	
\draw
  (.5,-.2) node {$\pi_3^f$}

  (0,.2451)--(.4301,.5698)--(.3247,.8149)
	(.4301,.5698)--(1,0)
  ;
	
\draw[thin]	
	(0.85,0.07) node {\footnotesize 1}
	(0.43,0.49) node {\footnotesize 2}
	(0.05,0.2) node {\footnotesize 3}
	(0.05,0.07) node {\footnotesize 4}

  (0.86,0.15)--(0.89,0.22)
	(0.92,0.22) node {\footnotesize 1}
	(0.47,0.61) node {\footnotesize 2}
  (0.35,0.78)--(0.38,0.85)
	(0.42,0.85) node {\footnotesize 3}
	(0.42,0.69) node {\footnotesize 4}

	(0.05,0.36) node {\footnotesize 1}
	(0.05,0.89) node {\footnotesize 2}
	(0.29,0.75) node {\footnotesize 3}
	(0.36,0.6) node {\footnotesize 4}
	;
  }}

  (9,0) node {\tikz[scale=2.5]{

\draw[thick]
  (0,0)--(1,0)--(.4301,.7548)--(0,1)--cycle
	(0,0) node[below] {1}
	(1,0) node[below] {2}
	(.4301,.7548) node[above right] {3}
	(0,1) node[above] {4}
	;
	
\draw
  (.5,-.2) node {$\pi_4^f$}

  (.5698,0)--(.2451,.7548)--(.7548,.3247)
	(.2451,.7548)--(0,1)
  ;
	
\draw[thin]	
	(0.05,0.07) node {\footnotesize 1}
	(0.48,0.07) node {\footnotesize 2}
	(0.21,0.71) node {\footnotesize 3}
	(0.05,0.87) node {\footnotesize 4}

	(0.6,0.07) node {\footnotesize 1}
	(0.39,0.56) node {\footnotesize 2}
	(0.71,0.29) node {\footnotesize 3}
	(0.87,0.07) node {\footnotesize 4}

	(0.41,0.7) node {\footnotesize 1}
  (0.65,0.43)--(0.68,0.5)
	(0.72,0.5) node {\footnotesize 2}
	(0.29,0.78) node {\footnotesize 3}
  (0.075,0.94)--(0.105,1.01)
  (0.145,1.01) node {\footnotesize 4}
	;
  }}

  (0,-4) node {\tikz[scale=2.5]{

\draw[thick]
  (0,0)--(1,0)--(.4301,.7548)--(0,1)--cycle
	(0,0) node[below] {1}
	(1,0) node[below] {2}
	(.4301,.7548) node[above right] {3}
	(0,1) node[above] {4}
	;
	
\draw
  (.5,-.2) node {$\pi_5^f$}

  (.7548,0)--(.5698,.4301)--(.8603,.1850)
	(.5698,.4301)--(0,1)
  ;
	
\draw[thin]	
	(0.05,0.07) node {\footnotesize 1}
	(0.05,0.87) node {\footnotesize 2}
	(0.53,0.39) node {\footnotesize 3}
	(0.67,0.07) node {\footnotesize 4}

	(0.79,0.07) node {\footnotesize 1}
	(0.71,0.24) node {\footnotesize 2}
	(0.81,0.16) node {\footnotesize 3}
	(0.87,0.07) node {\footnotesize 4}

	(0.41,0.7) node {\footnotesize 1}
	(0.3,0.77) node {\footnotesize 2}
	(0.59,0.47) node {\footnotesize 3}
  (0.77,0.28)--(0.8,0.35)
  (0.84,0.35) node {\footnotesize 4}
	;
  }}

  (3,-4) node {\tikz[scale=2.5]{

\draw[thick]
  (0,0)--(1,0)--(.4301,.7548)--(0,1)--cycle
	(0,0) node[below] {1}
	(1,0) node[below] {2}
	(.4301,.7548) node[above right] {3}
	(0,1) node[above] {4}
	;
	
\draw
  (.5,-.2) node {$\pi_6^f$}

  (0,.5698)--(.2451,.7548)--(.1850,.8945)
	(.2451,.7548)--(1,0)
  ;
	
\draw[thin]	
	(0.85,0.07) node {\footnotesize 1}
	(0.05,0.07) node {\footnotesize 2}
	(0.05,0.54) node {\footnotesize 3}
	(0.23,0.68) node {\footnotesize 4}

	(.85,.17)--(.88,.24)
	(0.91,0.24) node {\footnotesize 1}
	(0.4,0.7) node {\footnotesize 2}
	(.22,.85)--(.25,.92)
	(0.29,0.92) node {\footnotesize 3}
	(0.29,0.77) node {\footnotesize 4}

	(0.05,0.69) node {\footnotesize 1}
	(0.05,0.89) node {\footnotesize 2}
  (0.16,0.88)--(0.19,0.95)
	(0.19,1.01) node {\footnotesize 3}
  (0.18,0.78) node {\footnotesize 4}
	;
  }}

  (6,-4) node {\tikz[scale=2.5]{

\draw[thick]
  (0,0)--(1,0)--(.4301,.7548)--(0,1)--cycle
	(0,0) node[below] {1}
	(1,0) node[below] {2}
	(.4301,.7548) node[above right] {3}
	(0,1) node[above] {4}
	;
	
\draw
  (.5,-.2) node {$\pi_{14}^s$}

  (.7548,0)--(.1850,.3247)--(0,.2451)
	(.1850,.3247)--(.4301,.7548)
  ;
	
\draw[thin]	
	(0.5,0.07) node {\footnotesize 1}
	(0.05,0.07) node {\footnotesize 2}
	(0.05,0.19) node {\footnotesize 3}
	(0.18,0.25) node {\footnotesize 4}

	(0.435,0.66) node {\footnotesize 1}
	(0.26,0.35) node {\footnotesize 2}
	(0.77,0.07) node {\footnotesize 3}
	(0.87,0.07) node {\footnotesize 4}

	(0.05,0.89) node {\footnotesize 1}
	(0.34,0.73) node {\footnotesize 2}
	(0.15,.38) node {\footnotesize 3}
  (0.05,0.34) node {\footnotesize 4}
	;
  }}
	
	(9,-4) node {\tikz[scale=2.5]{

\draw[thick]
  (0,0)--(1,0)--(.4301,.7548)--(0,1)--cycle
	(0,0) node[below] {1}
	(1,0) node[below] {2}
	(.4301,.7548) node[above right] {3}
	(0,1) node[above] {4}
	;
	
\draw
  (.5,-.2) node {$\pi_{15}^s$}

  (.7548,0)--(.5698,.4301)--(.5698,.5698)
	(.5698,.4301)--(0,1)
  ;
	
\draw[thin]	
	(0.05,0.07) node {\footnotesize 1}
	(0.05,0.87) node {\footnotesize 2}
	(0.53,0.38) node {\footnotesize 3}
	(0.66,0.07) node {\footnotesize 4}

	(0.88,0.07) node {\footnotesize 1}
	(0.79,0.07) node {\footnotesize 2}
	(0.61,0.44) node {\footnotesize 3}
  (0.585,0.52)--(0.615,0.59)
	(0.655,0.59) node {\footnotesize 4}

  (0.075,0.94)--(0.105,1.01)
	(0.135,1.01) node {\footnotesize 1}
	(0.53,0.54) node {\footnotesize 2}
  (0.53,0.59)--(0.56,0.66)
  (0.6,0.66) node {\footnotesize 3}
	(0.4,.7) node {\footnotesize 4}
	;
  }}
	
	;

\end{tikzpicture}
\end{center}
\caption{The realizations of $3$-self-affinities of combinatorial type $\mathcal C$ of the quadrangle $Q[0.43015\ldots,0.75487\ldots]$.\label{fig:example}}
\end{figure}
Figure~\ref{fig:example} illustrates these realizations, all given for the natural representative of $Q[0.43015...,0.75487...]$ with vertices $\VECT{0}{0}$, $\VECT{1}{0}$, $\VECT{0.43015...}{0.75487...}$ and $\VECT{0}{1}$. The numbers indicate how the affine maps from the dissected quadrangle onto its pieces act on the vertices.


\section{Summarizing and illustrating all types of $3$-self-affine convex quadrangles\label{sec:union}}

\begin{thm}
\begin{itemize}
\item[(i)]
Every dissection of a convex quadrangle into $3$ convex quadrangles is based on one of the combinatorial types $\mathcal A$, $\mathcal B$ or $\mathcal C$.

\item[(ii)]
A convex quadrangle $Q$ is $3$-self-affine based on a dissection of combinatorial type $\mathcal A$ if and only if its unique parametrization $Q[x,y]$ with $(x,y) \in \mathcal P$ satisfies $(x,y) \in \mathbf T=\{(x,1): 0 < x \le 1\}$ (i.e., $Q$ is a trapezoid) or $(x,y) \in \mathbf A$ (see Theorem~\ref{thm:AB}(i)).

\item[(iii)]
A convex quadrangle $Q$ is $3$-self-affine based on a dissection of combinatorial type $\mathcal B$ if and only if its unique parametrization $Q[x,y]$ with $(x,y) \in \mathcal P$ satisfies $(x,y) \in \mathbf T\cup \mathbf B_1 \cup \mathbf B_2$ (see Theorem~\ref{thm:AB}(ii)).

\item[(iv)]
A convex quadrangle $Q$ is $3$-self-affine based on a dissection of combinatorial type $\mathcal C$ if and only if its unique parametrization $Q[x,y]$ with $(x,y) \in \mathcal P$ satisfies $(x,y) \in \mathbf T \setminus \{(1,1)\}$ (i.e., $Q$ is a trapezoid, but no parallelogram) or $(x,y) \in \mathbf C$ (see Theorem~\ref{thm:C}) or $(x,y)$ is one of the $13$ singular parameters given in Table~\ref{tab:1}.
\end{itemize}
\end{thm}
\begin{proof}
Glass-cut dissections correspond to types $\mathcal A$ and $\mathcal B$, and we have seen that non-glass-cut dissections result in type $\mathcal C$. This gives (i).
Claims (ii) and (iii) are obtained from Proposition~\ref{prop:trapezoidsAB} and Theorem~\ref{thm:AB}. Proposition~\ref{prop:trapezoidsC} and Theorem~\ref{thm:C} yield (iv).
\end{proof}

Figure~\ref{fig:summary} illustrates the $5$ families $\mathbf T,\mathbf A,\mathbf B_1, \mathbf B_2, \mathbf C \subseteq \mathcal P$ as well as the $13$ singular parameters from Table~\ref{tab:1}.
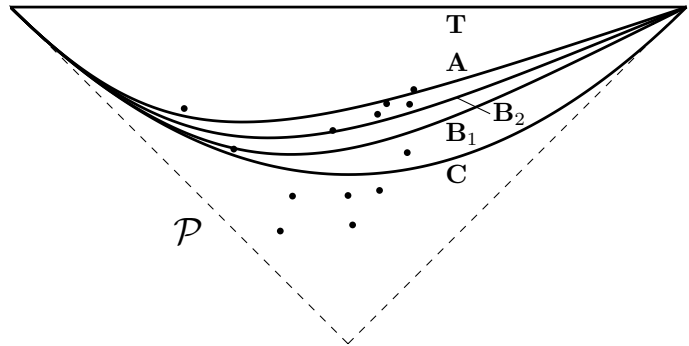
\begin{figure}
\begin{center}
\begin{tikzpicture}[scale=9cm/1cm]

\draw
  (0,1.0005) node[inner sep=0pt,below right] {\includegraphics[trim=2.215cm 17.75cm 6.375cm 6.92cm, clip, width=9cm]{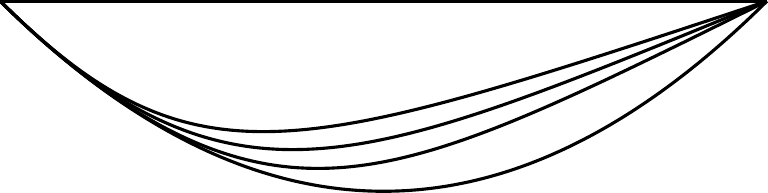}}
	;

\draw[dashed]
  (.5,.5)--(0,1) 
	(.5,.5)--(1,1)
	;

\draw[thin]
  (.3,.7) node[below left] {\LARGE $\mathcal P$}
  (.63,1) node[below right] {$\mathbf T$}
  (.63,.917) node[right] {$\mathbf A$}
  (.63,.81) node[right] {$\mathbf B_1$}
  (.66,.864)--(.71,.84) (.7,.84) node[right] {$\mathbf B_2$}
  (.63,.751) node[right] {$\mathbf C$}
	;

\fill
  (0.543689013,0.8392867552) circle (.005) 
  (0.557061798,0.8549068110) circle (.005) 
  (0.546602348,0.7266988258) circle (.005) 
  (0.506785038,0.6756719098) circle (.005) 
  (0.4775922500,0.8153009687) circle (.005) 
  (0.258055873,0.8478103848) circle (.005) 
  (0.587505999,0.7825751215) circle (.005) 
  (0.5000000000,0.719223594) circle (.005) 
  (0.5910090475,0.8540328194) circle (.005) 
  (0.4180334103,0.7183116995) circle (.005) 
  (0.3313309486,0.7878302113) circle (.005) 
  (.4,.6666666666) circle (.005) 
  (0.5971728861,0.8758670596) circle (.005) 
	;

\end{tikzpicture}
\end{center}
\caption{All parameters $(x,y) \in \mathcal P$ representing $3$-self-affine convex quadrangles $Q[x,y]$.\label{fig:summary}}
\end{figure}
In other words, a convex quadrangle $Q$ is $3$-self-affine if and only if there is an affine transformation mapping the vertices of $Q$ onto $\VECT{0}{1}$, $\VECT{0}{0}$, $\VECT{1}{0}$ and a point $\VECT{x}{y}$ displayed in Figure~\ref{fig:summary}.


\section{Non-convex quadrangles\label{sec:non-convex}}

Knowing that glass-cut self-affinity of quadrangles (and even of general topological discs) reduces to convex objects \cite[Theorem~1(i)]{richter2024+}, we consider it worth mentioning that there exist non-convex self-affine quadrangles. The positive result is also remarkable given the fact that there are no self-similar non-convex quadrangles \cite[Folgerung~3.2]{osburg2004}, see also \cite[Satz~6]{ditrich1995}.

\begin{prop}\label{prop:nonconvex_positive}
For every integer $n \ge 3$, there is an $n$-self-affine non-convex quadrangle.
\end{prop}

The proof is prepared by two lemmas. In the non-convex case we parametrize quadrangles $Q[x,y]$ as in Section~\ref{sec:parameters}, but with $\VECT{x}{y}$ representing the non-convex vertex. Then $0<x,y$ and $x+y<1$, and the only alternative parametrization of $Q[x,y]$ is $Q[y,x]$. 

\begin{lem}\label{lem:nonconvex1}
Let $Q[x,y]$ be non-convex and let $k \ge 2$ be an integer. Then the quadrangle $Q\left[\frac{1-(1-(x+y))^k}{x+y}(x,y)\right]$ admits a dissection into $k$ affine images of $Q[x,y]$.
\end{lem}

\begin{proof}
The affine map 
$$
\alpha\VECT{\xi}{\eta}=
\left(\begin{array}{cc}
1-x & -x\\
-y & 1-y
\end{array}\right)
\VECT{\xi}{\eta}+\VECT{x}{y}
$$
satisfies 
$$
\alpha^m\VECT{1}{0}=\VECT{1}{0}, \alpha^m\VECT{0}{1}=\VECT{0}{1},  \alpha^m\VECT{0}{0}=\frac{1-(1-(x+y))^m}{x+y}\VECT{x}{y}
$$
for $m=0,1,\ldots$, as can easily be seen by induction. Then $\alpha^m(Q[x,y])$ has the vertices $\VECT{0}{1}$, $\alpha^m\VECT{0}{0}$, $\VECT{1}{0}$ and $\alpha^m\VECT{x}{y}=\alpha^{m+1}\VECT{0}{0}$, and the quadrangle $Q\left[\frac{1-(1-(x+y))^k}{x+y}(x,y)\right]$ is dissected into $\alpha^m(Q[x,y])$, $m=0,\ldots,k-1$, see Figure~\ref{fig:nonconvex1}. 
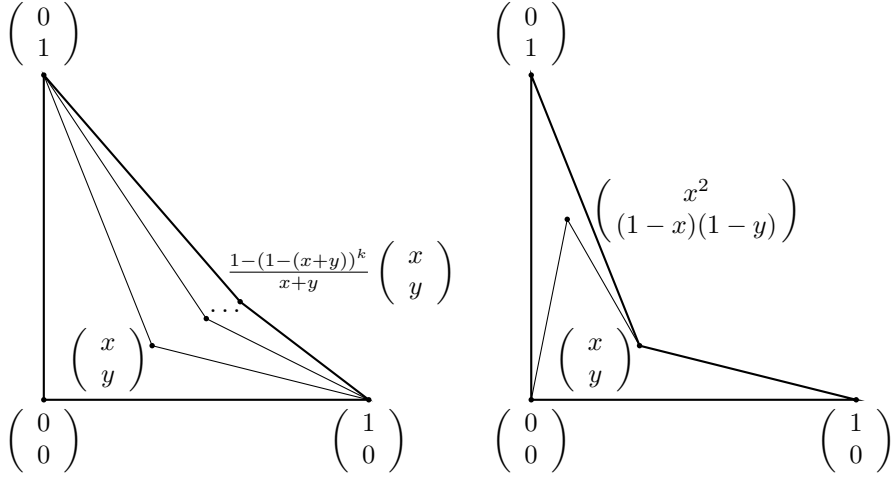
\begin{figure}
\begin{center}
\begin{tikzpicture}[xscale=.36,yscale=.36]

\draw[thick]
  (0,12)--(0,0)--(12,0)--(7.25,3.625)--cycle
	;

\draw
  (12,0)--(4,2)--(0,12)--(6,3)--cycle
	;

\fill
  (0,0) circle (.1) node[below] {$\VECT{0}{0}$} 
  (12,0) circle (.1) node[below] {$\VECT{1}{0}$} 
  (0,12) circle (.1) node[above] {$\VECT{0}{1}$} 
  (4,2) circle (.1) 
	(4.2,3) node[below left] {$\VECT{x}{y}$} 
  (6,3) circle (.1)  
	(6.775,3.3125) node {$\dots$}
	(7.25,3.625) circle (.1)
	(6.4,3.2) node[above right] {$\frac{1-(1-(x+y))^k}{x+y}\VECT{x}{y}$}
	;

\draw[thick]
	(18,0)--(30,0)--(22,2)--(18,12)--cycle
	;
	
\draw
	(18,0)--(19.33,6.67)--(22,2)
	;
	
\fill
	(18,0) circle (.1) node[below] {$\VECT{0}{0}$} 
  (30,0) circle (.1) node[below] {$\VECT{1}{0}$} 
  (18,12) circle (.1) node[above] {$\VECT{0}{1}$} 
  (22,2) circle (.1) 
	(22.2,3) node[below left] {$\VECT{x}{y}$}
	(19.33,6.67) circle (.1) 
	(19.9,7) node[right] {$\VECT{x^2}{\!\!(1-x)(1-y)\!\!}$}
  ;
	
\end{tikzpicture}
\end{center}
\caption{Dissections from Lemmas~\ref{lem:nonconvex1} and \ref{lem:nonconvex2}.\label{fig:nonconvex1}}
\end{figure}
\end{proof}

\begin{lem}\label{lem:nonconvex2}
Let $Q[x,y]$ be non-convex. Then $Q[x,y]$ admits a dissection into an affine image of $Q[x,y]$ and a remaining quadrangle with vertices $\VECT{0}{0}$, $\VECT{1}{0}$, $\VECT{x}{y}$ and $\VECT{x^2}{(1-x)(1-y)}$.
(We admit the last quadrangle to be convex or non-convex or to degenerate into a triangle, depending on the angle at $\VECT{x}{y}$.)
\end{lem}

\begin{proof}
The affine map 
$$
\beta\VECT{\xi}{\eta}=
\left(\begin{array}{cc}
x & 0\\
y-1 & -1
\end{array}\right)
\VECT{\xi}{\eta}+\VECT{0}{1}
$$
satisfies 
\begin{align*}
\beta\VECT{0}{1}=\VECT{0}{0},\quad & \beta\VECT{0}{0}=\VECT{0}{1},\\
\beta\VECT{1}{0}=\VECT{x}{y},\quad &\beta\VECT{x}{y}=\VECT{x^2}{(1-x)(1-y)}.
\end{align*}
We see that $Q[x,y]$ splits into $\beta(Q[x,y])$ and the remainder described in Lemma~\ref{lem:nonconvex2}, see Figure~\ref{fig:nonconvex1}.
\end{proof}

\begin{proof}[Proof of Proposition~\ref{prop:nonconvex_positive}]
We shall show that, for a particular choice of $(x,y)$, the quadrangle $Q\left[\frac{1-(1-(x+y))^{n-1}}{x+y}(x,y)\right]$, which has a dissection into $n-1$ affine images of $Q[x,y]$ by Lemma~\ref{lem:nonconvex1}, can be transformed by an affine map into the remaining quadrangle $R$ from Lemma~\ref{lem:nonconvex2}. Then $Q[x,y]$ from Lemma~\ref{lem:nonconvex2} is dissected into $n$ affine copies of itself. 

We proceed as follows. The polynomial 
$$
f_n(x)=\left(1-(1-x)^{2n-2}\right)\left(1-x+x^2\right)-x(2-x)
$$
and its derivative
$$
f_n'(x)= (2n-2)(1-x)^{2n-3}\left(1-x+x^2\right)+\left(1-(1-x)^{2n-2}\right)(-1+2x)-2+2x
$$
satisfy $f_n(0)=f_n(1)=0$, $f_n'(0)=2(n-2)>0$ and $f_n'(1)=1>0$, whence the intermediate value theorem gives us $x_0 \in (0,1)$ such that 
\begin{equation}\label{eq:zero}
f_n(x_0)=0.
\end{equation} 

Now we work with the particular non-convex quadrangle $Q[x,y]=Q[x_0,x_0(1-x_0)]$. Then 
$$
Q\left[\frac{1-(1-(x+y))^{n-1}}{x+y}(x,y)\right]=Q\left[\frac{1-(1-x_0)^{2n-2}}{2-x_0}(1,1-x_0)\right],
$$
and the vertices of the remaining quadrangle $R$ are
$$
\VECT{0}{0},\quad \VECT{1}{0},\quad \VECT{x}{y}=x_0\VECT{1}{1-x_0},
$$
$$
\VECT{x^2}{(1-x)(1-y)}=\VECT{x_0^2}{(1-x_0)\left(1-x_0+x_0^2\right)}.
$$

The affine map 
$$
\gamma\VECT{\xi}{\eta}=
\left(\begin{array}{cc}
x_0^2 & 1\\
(1-x_0)\left(1-x_0+x_0^2\right) & 0
\end{array}\right)
\VECT{\xi}{\eta}
$$
satisfies
$$
\gamma\VECT{0}{0}=\VECT{0}{0},\; \gamma\VECT{1}{0}=\VECT{x_0^2}{(1-x_0)\left(1-x_0+x_0^2\right)},\; \gamma\VECT{0}{1}=\VECT{1}{0}.
$$
For proving that $\gamma\left(Q\left[\frac{1-(1-x_0)^{2n-2}}{2-x_0}(1,1-x_0)\right]\right)=R$, it remains to verify that
$$
\gamma\left(\frac{1-(1-x_0)^{2n-2}}{2-x_0}\VECT{1}{1-x_0}\right)=x_0\VECT{1}{1-x_0}.
$$
This is
$$
\frac{1-(1-x_0)^{2n-2}}{2-x_0}\left(1-x_0+x_0^2 \right)\VECT{1}{1-x_0}=x_0\VECT{1}{1-x_0},
$$
and amounts to
$$
\left(1-(1-x_0)^{2n-2}\right)\left(1-x_0+x_0^2 \right)=x_0(2-x_0),
$$
which is nothing but \eqref{eq:zero}.

Figure~\ref{fig:nonconvex2} illustrates the cases $n=3$ (with $Q[0.43015...,0.24512...]$) and $n=4$ (with $Q[0.48662...,0.24982...]$).
\begin{figure}
\begin{center}
\begin{tikzpicture}[scale=4]

\draw
  (0,0) node{\tikz[scale=4]{
\draw[thick]
  (0,0)--(1,0)--(.43015,.24512)--(0,1)--cycle
	;
\draw
  (0,0)--(.18503,.43015)--(.43015,.24512)
  (.18503,.43015)--(.32471,.18503)--(1,0)
	;}}
	(1.4,0) node{\tikz[scale=4]{
\draw[thick]
  (0,0)--(1,0)--(.48662,.24982)--(0,1)--cycle
	;
\draw
  (0,0)--(.23680,.38512)--(.48662,.24982)
  (.23680,.38512)--(.36505,.18741)--(1,0)--(.46126,.23680)--cycle
	;}}
	;
	
\end{tikzpicture}
\end{center}
\caption{Realizations of Proposition~\ref{prop:nonconvex_positive} for $n=3,4$.\label{fig:nonconvex2}}
\end{figure}
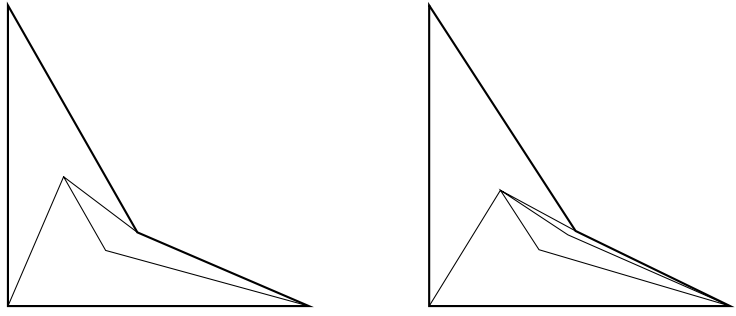
\end{proof}

\begin{cor}\label{cor:nonconvex}
For every integer $n \ge 3$, there is a non-convex quadrangle that is $n+k(n-1)$-self-affine for every $k=0,1,\ldots$
\end{cor}

\begin{proof}
Let $Q$ be $n$-self-affine. Since $Q$ has a dissection into $n$ affine copies of $Q$, we can split one of these copies again into $n$ affine copies of $Q$. This way $Q$ is dissected into $n+(n-1)$ affine copies. Iterating this procedure gives $n+k(n-1)$-self-affinities for all $k=0,1,\ldots$
\end{proof}

In Proposition~\ref{prop:nonconvex_positive} the case $n=2$ is excluded for the following reason.

\begin{prop}\label{prop:nonconvex_negative}
There is no $2$-self-affine non-convex quadrangle.
\end{prop}

Again we prepare the proof by a lemma.

\begin{lem}\label{lem:diagonals}
Let $Q_0$ and $Q$ be non-convex quadrangles of the same affine type such that $Q_0 \subseteq Q$. If $Q_0$ has the same inner or outer diagonal as $Q$ then $Q_0=Q$.

In particular, if $Q_0$ is a tile in a dissection of $Q$ representing an $n$-self-affinity of $Q$, $n \ge 2$, then neither the inner nor the outer diagonal of $Q_0$ coincides with the respective diagonal of $Q$. 
\end{lem}

\begin{proof}
For two different points $a,b \in \mathbb R^2$, let $ab$ denote the line segment with endpoints $a$ and $b$, and let $\AFF(a,b)$ stand for the straight line through $a$ and $b$.

Let the vertices of $Q$ be denoted by $v_1,v_2,v_3,v_4$ such that $v_2v_4$ is the outer diagonal and $v_3$ is the non-convex vertex. Let $d$ be the common point of $v_2v_4$ with $\AFF(v_1,v_3)$. Then the ratio of lengths $I(Q)=\frac{|v_1v_3|}{|v_3d|}$ is invariant under affine transformations $\alpha$, as they preserve ratios of lengths of parallel segments. That is, $I(\alpha(Q))=I(Q)$. In particular, $I(Q_0)=I(Q)$.

If $Q_0$ with $Q_0 \subseteq Q$ has the same inner diagonal $v_1v_3$ as $Q$, its outer diagonal is $v_{2,0} v_{4,0}$ with, say, $v_{2,0}$ in the triangle 
$\triangle v_1 v_2 v_3$ and hence $v_{4,0}$ in $\triangle v_1 v_3 v_4$. Then the common point $d_0$ of $v_{2,0} v_{4,0}$ and $\AFF(v_1,v_3)$ satisfies $d_0 \in v_3 d \setminus \{v_3\}$ with $d_0 = d$ if and only if $v_{2,0}v_{4,0}=v_2v_4$; i.e., if and only if $Q_0=Q$. Consequently, $I(Q_0)=\frac{|v_1v_3|}{|v_3d_0|} \ge \frac{|v_1v_3|}{|v_3d|}=I(Q)$ with equality if and only if $d_0=d$, whence $I(Q_0)=I(Q)$ gives $Q_0=Q$.

Now let $Q_0$ and $Q$ share their outer diagonal $v_2v_4$, and let $v_{1,0}$ and $v_{3,0}$ be the additional vertices of $Q_0$, $v_{3,0}$ being the non-convex one. The equality $I(Q_0)=I(Q)$ and the intercept theorem yield 
\begin{equation}\label{eq:ratio}
\frac{\DIST(v_{1,0}, \AFF(v_2,v_4))}{\DIST(v_{3,0}, \AFF(v_2,v_4))}= \frac{\DIST(v_1, \AFF(v_2,v_4))}{\DIST(v_3, \AFF(v_2,v_4))}.
\end{equation}
The inclusion $Q_0 \subseteq Q$ shows that $v_{1,0}$ and $v_{3,0}$ are within the convex quadrangle $v_1 d_2 v_3 d_4$, $d_2$ and $d_4$ being the intersections of $v_1v_2$ with $\AFF(v_3,v_4)$ and of $v_1v_4$ with $\AFF(v_2,v_3)$, respectively. In particular,
$$
\frac{\DIST(v_{1,0}, \AFF(v_2,v_4))}{\DIST(v_{3,0}, \AFF(v_2,v_4))}\le \frac{\DIST(v_1, \AFF(v_2,v_4))}{\DIST(v_{3,0}, \AFF(v_2,v_4))}
\le \frac{\DIST(v_1, \AFF(v_2,v_4))}{\DIST(v_3, \AFF(v_2,v_4))}$$
with equality if and only if $v_{1,0}=v_1$ and $v_{3,0}=v_3$. So \eqref{eq:ratio} implies $Q_0=Q$.
\end{proof}

\begin{proof}[Proof of Proposition~\ref{prop:nonconvex_negative}]
Assume that a non-convex quadrangle $Q$ is dissected into two affine copies $Q_1$ and $Q_2$ of $Q$. Then $\Gamma=Q_1 \cap Q_2$ is a polygonal arc connecting two different points $p_j$, $j=1,2$, of the boundary of $Q$. The arc $\Gamma$ is not a line segment, because both $Q_1$ and $Q_2$ are non-convex. Let $i \ge 1$ be the number of its inner vertices.

We count the total number of vertices of $Q_1$ and $Q_2$ with multiplicities; i.e., a common vertex of both $Q_1$ and $Q_2$ is considered as a double vertex. Every vertex of $Q$ is at least a single vertex of $Q_1$ and $Q_2$. If some $p_j$ is a convex vertex of $Q$, we win a new vertex in so far as it becomes a double vertex. If $p_j$ is the non-convex vertex of $Q$ then we may win a new vertex, but do not need, because a straight angle could appear. If $p_j$ is not a vertex of $Q$, $p_j$ is a double new vertex. Moreover, every inner vertex of $\Gamma$ gives two new vertices of $Q_1$ and $Q_2$. So the total number $8$ of vertices of $Q_1$ and $Q_2$ is $4$ (vertices of $Q$) plus $2i$ (inner vertices of $\Gamma$) plus the vertices won in $p_1$ and $p_2$. As the last summand is at least one, we get $i=1$ and exactly two new vertices are won in $p_1$ and $p_2$.

It is impossible to win no new vertex in $p_1$ and two in $p_2$, because then neither of $Q_1$ and $Q_2$ had a non-convex vertex in the non-convex vertex $p_1$ of $Q$, and the only non-convex vertex of $Q_1$ and $Q_2$ would appear in the inner vertex of $\Gamma$, so that $Q_1$ and $Q_2$ had only one non-convex vertex in total.

Hence both $p_1$ and $p_2$ are vertices of $Q$ and both are vertices of both of $Q_1$ and $Q_2$. Moreover, $p_1$ and $p_2$ are opposite vertices of $Q$, because both $Q_1$ and $Q_2$ are quadrangles. But then $p_1$ and $p_2$ are either the endpoints of the outer diagonal of $Q$, whence both $Q_1$ and $Q_2$ have the same outer diagonal as $Q$, or $p_1$ and $p_2$ are the non-convex vertex of $Q$ and its opposite vertex, whence one of $Q_1$ and $Q_2$ shares the inner diagonal with $Q$. However, both situations contradict Lemma~\ref{lem:diagonals}.
\end{proof}

The previous results give a first access to the non-convex case, but leave many questions open. Is every non-convex quadrangle self-affine? Is there some non-convex quadrangle $Q$ and an integer $n_0=n_0(Q)$ such that $Q$ is $n$-self-affine for every integer $n \ge n_0$? Does the last apply to every non-convex $Q$? Does the last apply to every non-convex $Q$ with $n_0$ not depending on $Q$, as it does for convex $Q$ with $n_0=5$ \cite[Theorem~2]{richter2024+}?



\end{document}